\documentclass{amsart}
\usepackage{amsmath}
\usepackage{amssymb}
\usepackage{mathrsfs}
\usepackage{pgf,tikz}
\usepackage{hyperref}
\usetikzlibrary{positioning}
\usetikzlibrary{shadings}
\usetikzlibrary{arrows,automata}
\usepackage{graphicx}
\usepackage{systeme,mathtools}
\usepackage{amsthm,enumitem}
\usepackage{systeme}
\input xy
\xyoption{all}

\newtheorem{theorem}{Theorem}[section]

\newtheorem{lemma}[theorem]{Lemma}
\newtheorem{corollary}[theorem]{Corollary}
\newtheorem{proposition}[theorem]{Proposition}
\theoremstyle{definition}
\newtheorem{definition}[theorem]{Definition}
\newtheorem{remark}{Remark}

\begin{document}
	
	\title[Almost locally simple artinian algebras with involutions]{Subnormal subgroups of almost locally simple artinian algebras with involutions}
	
	\author[Dau Thi Hue]{Dau Thi Hue}
	\address{Faculty of Mathematics and Computer Science, University of Science, Ho Chi Minh City, Vietnam} 
	\address{Vietnam National University, Ho Chi Minh City, Vietnam}
	\email{dthue.sgddt@tphcm.gov.vn}

	\author[Huynh Viet Khanh]{Huynh Viet Khanh}
	\address{Department of Mathematics and Informatics, HCMC University of Education, 280 An Duong Vuong Str., Dist. 5, Ho Chi Minh City, Vietnam}
	\email{khanhhv@hcmue.edu.vn}
	
	\author[Bui Xuan Hai]{Bui Xuan Hai}
	\address{Faculty of Mathematics and Computer Science, University of Science, Ho Chi Minh City, Vietnam} 
	\address{Vietnam National University, Ho Chi Minh City, Vietnam}
	\email{bxhai@hcmus.edu.vn}

	\begin{abstract} 
		In this paper, we investigate subnormal subgroups of the multiplicative group of an almost locally simple artinian algebra with involution. In particular, we show that if either the set of traces or the set of norms of such a subgroup with respect to this involution is central, then the  algebra must be either a quaternion division algebra or the matrix ring of degree $2$ over a field. 
	\end{abstract}
	\maketitle
	
	\section{Introduction}
	
	Let $A$ be an associative ring with identity $1\ne 0$. Recall that an \textit{involution} on $A$ is a map  $^{\star}:A\to A$ which satisfies the following conditions for all elements $x$ and $y$ in $A$:
	
	\begin{enumerate}
		\item[(a)] $(x+y)^{\star}=x^{\star}+y^{\star}$;
		\item[(b)] $(xy)^\star=y^\star x^\star$;
		\item[(c)] $(x^\star)^\star=x$.
	\end{enumerate}
	Let $A$ be such a ring, and let $Z(A)$ denote the center of $A$. It is straightforward to check that $Z(A)$ is preserved under the involution $\star$. The restriction of $\star$ to $Z(A)$ is therefore an automorphism which is either the identity or of order $2$. Accordingly, an involution which leaves  the elements of $Z(A)$ fixed is called an \textit{involution of the first kind}. An involution whose restriction to the center is an automorphism of order $2$ is called \textit{involution of the second kind}. Throughout this paper, we only investigate rings with involutions of the first kind. 
	
	Recall that the set of \textit{symmetric} elements, the set of \textit{traces}, and the set of \textit{ norms} of $A$ are defined as follows:
	$$
	\begin{aligned}
		& A_+=\{x\in A \;|\; x^\star = x\},\\
		& T_A=\{x+x^\star\;|\; x\in A\},\\
		& N_A=\{xx^\star\;|\; x\in A\}.
	\end{aligned}
	$$
	Clearly that the sets $T_A$ and $N_A$ both are contained in $A_+$. 

In a series of his papers  (see e.g. \cite{Chacron.2016a.paper}-\cite{Chacron.2017b.paper}),
Chacron has considered the following  conditions: 
\begin{enumerate}
	\item[(C1)] $x^\star x=xx^\star$, for all $x\in A$.
	\item[(C2)] For each $x$ in $A$, there exists a positive integer $N$ depending on $x$ such that $\mathrm{d}_x^N(x^\star)=0$, where $\mathrm{d}_x:A\to A$ is a map given by $y\mapsto yx-xy$, and $\mathrm{d}_x^N$ is the $N$-th power of $\mathrm{d}_x$ under composition.
	\item[(C3)] For each $x$ in $A$, there exists a positive integer $N$ depending on $x$ such that $x^\star x^N =x^Nx^\star$.
	\item[(C4)] $x+x^\star$ is central for all $x\in A$.
	\item [(C5)] $xx^\star$ is central for all $x\in A$.
\end{enumerate}

In accordance with \cite[Propostion 2.6]{knus-merkurjev-rost-tignol-98}, if the involution $\star$ satisfies (C1), (C4) or (C5), then it is said to be \textit{commuting}, \textit{symplectic} or \textit{scalar} respectively, while (C2) and (C3) are called \textit{local power commuting condition} and \textit{local Engel condition} respectively. The relationships between these conditions are imposed on following implications which are obvious:
$$
{\rm (C5)}\Rightarrow{\rm (C4)}\Rightarrow{\rm (C1)}\Rightarrow{\rm (C2)} \text{ and } {\rm (C3)}.
$$

Over the past few years, there have been many works devoted to the study of certain rings with involutions satisfying some of these conditions. It was shown in \cite[Theorem 4.7]{Chacron.2017a.paper} that if $A$ is a semiprime ring satisfying a polynomial identity, then the conditions (C1), (C2), and (C3) are equivalent. If $A$ is a semiprime ring instead, then it was proved in \cite[Theortem 2.3]{Chacron.2016a.paper} that (C1), (C4) and (C5) are equivalent. At the other extreme, it was shown in \cite[Theorem 2.13]{Chacron.2017b.paper} that in case $A$ is a non-commutative simple ring which is algebraic over its center, then (C2) is equivalent to (C4) and $A$ must be a quaternion algebra over its center.  

There are close relationships between the conditions introduced by Chacron and the conditions on the set $T_A$ of traces and on the set $N_A$ of norms of $A$. It is clear that $\star$ satisfies (C4) if and only if $T_A\subseteq Z(A)$, while $\star$ satisfies (C5) if and only if $N_A\subseteq Z(A)$. Thus, in view of \cite[Theorem 2.13]{Chacron.2017b.paper}, if $A$ is non-commutative simple ring algebraic over $Z(A)$ such that $T_A\subseteq Z(A)$, then $A$ is a quaternion algebra over its center. This fact particularly says, in some sense, that the condition $T_A\subseteq Z(A)$, or equivalently that $\star$ satisfies (C4),  is really a strong condition which imposes a special structure on the whole $A$. 

This paper is a sequel of \cite{bien-hai-hue}, where the  structures of subgroups of certain algebras with involutions were investigated. Let $G$ be a subgroup of the multiplicative subgroup $A^\times$ of a ring $A$ with involution $\star$. We denote by $T_G$ and $N_G$ the set of traces and the set of norms in $G$, respectively; that is, 
$$
\begin{aligned}
	T_G=\{x+x^\star\;|\; x\in G\} \text{ and } N_G=\{xx^\star\;|\; x\in G\}.
\end{aligned}
$$
In \cite{bien-hai-hue}, instead of considering the two sets $T_A$ and  $N_A$, the authors have considered much smaller subsets $T_{G}$ and $N_{G}$, where $G$ is a non-central normal subgroup of $A^\times$. The main results of \cite{bien-hai-hue} demonstrates that the sets $T_G$ and $N_G$ are too big in the whole $A$; that is, if we impose suitable conditions on $T_G\cup N_G$, then the algebraic structure of the whole ring $A$ is effected. More precisely, it was shown that if $A$ is a division ring and $T_G\cup N_G$ is commutative, then $A$ is a quaternion division algebra over $Z(A)$ and $\star$ is of the symplectic type (see \cite[Theorem 5.5]{bien-hai-hue}). Moreover, if $A$ is a simple artinian ring and $T_G\cup N_G$ is contained in $Z(A)$, then either $A$ is a quaternion division algebra over $Z(A)$ or $A$ is isomorphic to the matrix ring ${\rm M}_2(Z(A))$ and $\star$ is of the symplectic type (see \cite[Theorem 6.1]{bien-hai-hue}). In this paper, we investigate a more general setting: the ring $A$ is assumed to be an \textit{almost locally simple artinian algebra} and $G$ is assumed to be an \textit{subnormal subgroup} of $A^\times$. It turns out that in this case subnormal subgroups of $A^\times$ also strongly reflect the multiplicative structure of $A^\times$. The typical result of this kind is Theorem \ref{theorem_subnormal skew linear group} showing how ``big" subnormal subgroups are in $A^\times$.

The paper is organized as follows. Apart from the Introduction, the paper is divided into 4 sections. In Section 2, we recall the definitions of some classes of division rings including the classes of centrally finite, locally finite, weakly locally finite division rings. We also introduce the notion of \textit{almost locally simple artinian algebra} which is the main object of the current paper. Some relationships among these classes of rings are also presented in the chart at the end of this section. This chart shows that the class of almost locally simple artinian algebras is too large and contains many important classes of algebras. In Section 3, we give some criterion for the restriction of an involution of a ring to a subring to be still an involution. Although this section is made for further use, it is interesting in it own right. Section 4 is devoted to the study of the matrix ring ${\rm M}_n(D)$  over a division ring $D$ with involution. The main result of this section is Theorem \ref{theorem_subnormal skew linear group} in which we give some equivalent conditions on the sets of traces $T_G$ and norms $N_G$, where $G$ is a non-central subnormal subgroup of ${\rm GL}_n(D)$, and these conditions, in some sense, are very strong implying a special construction for the whole ${\rm M}_n(D)$. It can be seen that Theorem \ref{theorem_subnormal skew linear group} broadly extends \cite[Theorem 6.1]{bien-hai-hue}.
In Section~5, we try to extend the results obtained in Section~4 for the class of almost locally simple artinian algebras. The primary result of this section is Theorem \ref{theorem_main-locally simple artin}. Finally, we include in this paper Corollaries \ref{corollary_subnormal skew linear group} and \ref{corollary_locally simple artin}, which generalizes some results of Chacron in the papers \cite{Chacron.2016a.paper}-\cite{Chacron.2017b.paper}.
Throughout this paper, for a ring $A$, we denote by $A^\times$ and $Z(A)$ respectively the multiplicative group and the center of $A$. Let $S$ be a subset of $A$. We say that $S$ is \textit{central} if $S\subseteq Z(A)$. In particular, an element $x\in A$ is said to be \textit{central} if $x\in Z(A)$. Also, we note  the following symbols and notations we use frequently in this paper:
\begin{enumerate}
	\item $\overline{S}$ is the subring of $A$ generated by $S$.
	\item  $S^\star=\{s^\star|s\in S\}$ if $\star$ is an involution on $A$.
	\item $C_A(S)=\{a\in A\mid as=sa \text{ for any } s\in S\}$ is the centralizer of $S$ in $A$.
\end{enumerate}
    If $A=D$ is a division ring, and $K$ is a division subring of $D$, then $K[S]$ and $K(S)$ are the subring and the division subring of $D$ generated by the set  $S\cup K$, respectively. We say that $K[S]$  and $K(S)$ are the \textit{subring} and the \textit{division subring}  of $D$ generated by $S$ over $K$, respectively. Finally, $K$ is said to be \textit{$S$-invariant} if $s^{-1}Ks\subseteq K$ for any $0\ne s\in S$. 

\section{Some classes of algebras}
	It can be seen from what we have discussed above, the aim of this paper is to introduce new class of simple algebras over fields,  and then investigate their structure when they are equipped an involution. 	
	\subsection{Weakly locally finite division rings}
	We begin with recalling some classes of division rings. A division ring is called \textit{centrally finite} if it is a finite-dimensional vector space over its center. Otherwise, it is called \textit{centrally infinite}. Centrally finite division rings naturally appear in many contexts. For example, the well-known Weddernburn-Artin Theorem says that every central simple algebra is isomorphic to the full matrix ring over a centrally finite division ring. At the other extreme, centrally infinite division rings are subdivided into some types. Let $D$ be a division ring with center $F$. Viewing $D$ as a vector space over $F$, we use the notation $[D:F]$ to indicate the dimension of this vector space. We say that $D$ is \textit{locally finite} if for every finite subset $S$ of $D$, the division subring $F(S)$ generated by $S$ over $F$ in $D$ is a finite dimensional vector space over $F$, that is, if $[F(S):F]<\infty$.  For a single element $a\in D$, it can be seen that $F(a)$ is then a subfield of $D$ containing $F$. Accordingly, an element $a\in D$ is said to be \textit{algebraic} over $F$ if there exists a non-zero polynomial of the polynomial ring $F[t]$ having $a$ as its root or equivalently, if  $F(a)$ is a finite extension over $F$, that is, if  $[F(a):F]<\infty$. If every element $a\in D$ is algebraic over $F$, then we say that $D$ is \textit{algebraic} over $F$ (briefly, $D$ is \textit{algebraic}). Clearly, every locally finite division ring is algebraic. However, it has been unknown that whether the converse statement holds or not. The problem of finding a suitable algebraic division ring which is not locally finite remains unsolved until now and it is often referred to as the Kurosh Problem for division rings \cite[Problem~ K]{Pa_Kurosh_1941}. Recall that a division ring $D$ is said to be \textit{weakly locally finite} if for every finite subset $S$ of $D$, the division subring $\mathbb{P}(S)$ of $D$ generated by $S$ over the prime subfield $\mathbb{P}$ of $D$ is centrally finite (see  \cite{Pa_dbh_2012}). The following lemma provides an important property of  weakly locally finite division rings.
	
	\begin{lemma}\label{lemma_property of weakly}
		Let $D$ be a division ring with center $F$, and $\mathbb{P}$ the prime subfield of $D$. Then, $D$ is weakly locally finite if and only if $K(S)$ is centrally finite for every finite subset $S$ and every subfield $K$ of $D$ such that $\mathbb{P}\subseteq K\subseteq F$.
	\end{lemma}
	
	\begin{proof}
		We begin by assuming that $D$ is weakly locally finite. We wish to show that $D_2=K(S)$ is a centrally finite division ring for every finite subset $S$ of $D$ and every central subfield $K$ of $D$ containing the prime subfield $\mathbb{P}$. If we set $D_1=\mathbb{P}(S)$, then this is a division subring of $ D_2$. Let $F_1$ and $F_2$ be the center of $D_1$ and $D_2$, respectively. One may simply check that $L:=C_{D_2}(D_1)$ is a subfield of $D_2$ containing $F_1$ and that $F_1=D_1\cap L$. If we put $R=D_1\otimes_{F_1}L$, then we have $Z(R)=F_1\otimes_{F_1}L=L$. By virtue of our hypothesis, we conclude that $D_1$ is finite dimensional over $F_1$, and so $n=[D_1:F_1]$ is finite. It follows immediately that 
		$$R\leq D_1\otimes_{F_1}D_1^{op}\otimes_{F_1}L\cong {\rm M}_n(F_1)\otimes_{F_1}L.$$ 
		Then, \cite[\S 7, Corollary 6]{Bo_Draxl_1983} implies that $[R:L]\leq [{\rm M}_n(F_1)\otimes_{F_1}L:L]<\infty$.
		
		Next, we show that $R=D_2$. Indeed, observe that the map $D_1\times L\to D_2$ defined by  $(x,y)\mapsto xy$, for any $(x,y)\in D_1\times L$, is clearly  $F_1$-bilinear. Accordingly, the map $f:R\to D_2: x\otimes y\mapsto xy$ is well defined by the universal property of tensor product. Even more, $f$ is actually a non-zero ring homomorphism. The simplicity of $R$ assures us to conclude that $f$ is injective. This forces the domain $R$ to be finite dimensional over its center $L$ and, hence, $R$ must be a division ring. Since $R$ contains both $K$ and $S$, it follows that $R=D_2$, and so $L=F_2$. In other words, $D_2$ is a centrally finite division ring, which proves the ``only if'' part. The converse of the proposition is fairly obvious.
	\end{proof}
	
	In view of Lemma \ref{lemma_property of weakly}, it is easy to see that every locally finite division ring is weakly locally finite. Indeed, assume that $D$ is a locally finite division ring with center $F$, $K$ is a subfield of $F$ containing the prime subfield $\mathbb{P}$, and $S$ is a finite subset of $D$. Note that any division subring of a centrally finite division ring is also centrally finite (see \cite[Theorem 3]{Pa_hn_2013}). Now, since $[F(S):F]<\infty$, it follows that $F(S)$ is  centrally finite, and so is $K(S)$.  
	
	Examples of weakly locally finite division rings which are not algebraic (so are not locally finite)  were provided in \cite{Deo.Bien.Hai.2019.paper}, and we refer  to this paper for more details.  In \cite[Theorem 1]{Deo.Bien.Hai.2019.paper}, it was proved that a division ring is weakly locally finite if and only if it is locally PI. In view of this theorem, it is appropriate to mention the examples presented in \cite{Pre_Wad}. 
	\subsection{Almost locally simple PI algebras}
	Because a centrally finite division ring is certainly a central simple algebra, the observation just made says that a weakly locally finite division ring has a further property that every its finite subset generates over the center a central simple subalgebra (finite-dimensional). This motivates us to introduce the new notion of \textit{almost locally central simple algebra}. Before stating the definition, let us make some observation. Let $A$ be an algebraic structure (e.g. a group, a ring, etc.), and $\mathcal{P}$,   be some algebraic property. As widely accepted tradition, one says that $A$ is \textit{locally $\mathcal{P}$} if every finitely generated substructure of $A$ satisfies the property $\mathcal{P}$. To avoid the confusion, we will say that $A$ is \textit{almost locally $\mathcal{P}$} if every finitely generated substructure of $A$ is contained in some substructure of $A$ which satisfies the property $\mathcal{P}$. The definitions also suggest that ``being almost locally $\mathcal{P}$ is much weaker than ``being  locally $\mathcal{P}$''.
	
	Now, let $K$ be a field. Recall that a $K$-algebra $A$ is called a \textit{central simple $K$-algebra} if $Z(A)=K$ and $A$ is a simple ring (see \cite{Bo_Herstein_1971}).  Note that in \cite{Bo_Draxl_1983}, the expression ``a central simple $K$-algebra" is used to indicate a $K$-algebra $A$, where $A$ is a simple algebra having $K$ as its center, and $[A:K]<\infty$. In \cite{Bo_Herstein_1971}, such an algebra is said to be a \textit{finite-dimensional central simple algebra}. It is well-known that an algebra $A$ over a field $K$ is finite-dimensional central simple if and only if $A$ is isomorphic as a $K$-algebra to the matrix ring  $\mathrm{M}_n(D)$ for some positive integer $n$ and some centrally finite division $D$ with center $K$. It is important to note that in this paper, we \textit{always follow} \cite{Bo_Herstein_1971} for the definition of central simple $K$-algebras.
	
	If a simple $K$-algebra $A$ satisfies a polynomial identity, then we say that $A$ is a \textit{simple PI-algebra} over $K$. Observe that if $A$ is a finite-dimensional central simple $K$-algebra, then $A$ satisfies a polynomial identity. Indeed, since $A\cong \mathrm{M}_n(D)$, where $D$ is a centrally finite division ring, so if we put $m=[D:Z(D)]$, then  $A$ is embedded in $\mathrm{M}_{nm}(Z(D))$. Hence, according to Amitsur-Levitzki's Theorem (see \cite[Theorem~ 1]{amitsur-levitzki}), $A$ satisfies the standard polynomial $s_{2nm}$. Conversely, assume that $A$ is a simple PI-algebra over a field $K$. By \cite[Theorem~ 1]{Pa_Kaplansky_1948}, $[A:Z(A)]<\infty$, which implies that $A$ is artinian. By the Wedderburn-Artin  Theorem, we have $A\cong \mathrm{M}_n(D)$ for some positive integer $n$ and a division ring $D$. Since $Z(A)=Z(D)$ and $[A:Z(A)]<\infty$, it follows that $D$ is centrally finite. The discussion above shows that the following proposition holds.
	\begin{proposition}\label{simple PI}
		An algebra over a field  is finite-dimensional central simple   if and only if it is simple PI.
	\end{proposition}
	Now, we are ready to give the following definitions.
	\begin{definition}\label{almost locally central simple}
		An algebra $A$ over a field $K$ is said to be \textit{almost locally (finite-dimensional) central simple} if every finitely generated $K$-subalgebra of $A$ is contained in a (finite-dimensional) central simple algebra over some field. 
	\end{definition}

	It is obvious that a (finite-dimensional) central simple algebra over a field $K$ is an almost locally (finite-dimensional) central simple algebra over $K$. We provide also some non-trivial examples of an almost locally finite-dimensional central simple algebra.
	\begin{proposition}\label{example_locally central}
		Let $D$ be a weakly locally finite division ring with center $K$ and $n$ a positive integer. Then, the matrix ring ${\rm M}_n(D)$ is an almost locally finite-dimensional central simple algebra over $K$.
	\end{proposition}
	\begin{proof}
		For any positive integer $k$, let $A_1, A_2,\dots, A_k$ be matrices in ${\rm M}_n(D)$, and $A$ be the $F$-algebra of ${\rm M}_n(D)$ generated by these matrices. Let $E$ be the division subring of $D$ generated by all entries of all matrices $A_1, A_2,\dots, A_k$. Since $D$ is weakly locally finite, $E$ is finite dimensional over its center, say $K$. Then, $A$ is contained in   ${\rm M}_n(E)$ which is a finite-dimensional central simple algebra over $K$. 
	\end{proof}
		
	\begin{definition}  
		An algebra $A$ over a field $K$ is called an \textit{almost locally simple PI-algebra} if every finitely generated $K$-subalgebra of $A$ is contained in a simple PI-algebra over some field. 
	\end{definition}

	\begin{theorem}\label{almost locally PI}
		An algebra $A$ over a field $K$ is an almost locally simple PI-algebra    if and only if $A$ is an almost locally finite-dimensional central simple algebra.
	\end{theorem}
	\begin{proof}
		Assume that $A$ is an almost locally simple PI-algebra over a field $K$. Let $B$ be a finitely generated $K$-subalgebra of $A$. Then, $B$ is contained in a simple PI-algebra, say $C$. By Proposition \ref{simple PI}, $C$ is a finite-dimensional central simple algebra. Hence, we conclude that $A$ is an almost locally finite-dimensional central simple algebra. 
		
		Conversely, assume that $A$ is an almost locally finite-dimensional central simple algebra over a field $K$, and $B$ is a finitely generated $K$-subalgebra of $A$. Then, $B$ is contained in some finite-dimensional central algebra $C$, which is a simple PI-algebra in view of Proposition \ref{simple PI}. Therefore, $A$ is an almost locally simple PI-algebra. 
	\end{proof} 
	\subsection{Locally central simple and finite dimensional algebras}
	In \cite{bar-on-gilat-matzri-vishne}, T. Bar-On, S. Gilat, E. Matzri and U. Vishne introduced the concept of a locally finite central simple and finite dimensional algebra. In what follows, we will make some observations to compare the class of such algebras with those of algebras we just introduced above. For the convenience of readers, we present the definition of a locally finite central simple and finite dimensional algebra here.
	\begin{definition}[{\cite[Definition 3.1]{bar-on-gilat-matzri-vishne}}]\label{definition_C_F}
		An algebra over a field $K$ is \textit{locally central simple and finite dimensional} if every finitely generated $K$-subalgebra is contained in a central simple algebra which is finite dimensional over $K$.  The class of such algebras is denoted by $\mathcal{C}_K$.
	\end{definition}
	
	In the same paper, the authors also provided a systematic study of such algebras. In particular, they proved in \cite[Corollary 3.2]{bar-on-gilat-matzri-vishne} that a $K$-algebra $A$ is locally finite central simple and finite dimensional if and only if it is a direct limit of finite dimensional $K$-algebras, giving an important characteristic of such algebras. It is clear from Definition \ref{definition_C_F} that the class of almost locally finite-dimensional central simple algebra over a field $K$ contains the class $\mathcal{C}_K$. Moreover, we have the following result.  
	\begin{proposition}\label{proposition_subclass}
		The class $\mathcal{C}_K$ is properly contained in the class of almost locally finite-dimensional central simple algebras over $K$.
	\end{proposition}
	\begin{proof}
		 Observe that there exists a weakly locally finite division ring which is not locally finite as noted above. So, such a division ring is almost locally finite-dimensional central simple but it does not belong to $\mathcal{C}_K$.
	\end{proof}
	\subsection{Almost locally simple artinian $K$-algebras} By definition, an algebra over a field $K$ is \textit{almost locally simple artinian} if every its finitely generated $K$-subalgebra is contained in a simple artinian algebra over some field. The class of such algebras contains the class of almost locally finite-dimensional central simple algebras. In view of Proposition~ \ref{example_locally central} and Proposition \ref{proposition_subclass}, we see that the last class appears very naturally and it is very large. This class indeed contains the important class of central simple finite-dimensional algebras and the last class contains the class $\mathcal{C}_K$ which was introduced and studied in \cite{bar-on-gilat-matzri-vishne}.
	Now, we summarise the relationships among theses classes of algebras in the following chart.
	
	\bigskip
	
			\begin{center}
			\tikzset{ 
				empty/.style ={draw, white}}
			\tikzstyle{block} = [rectangle, draw, 
			text width=5em, text centered, rounded corners, minimum height=4em]
			\tikzstyle{line} = [draw, -latex']
			
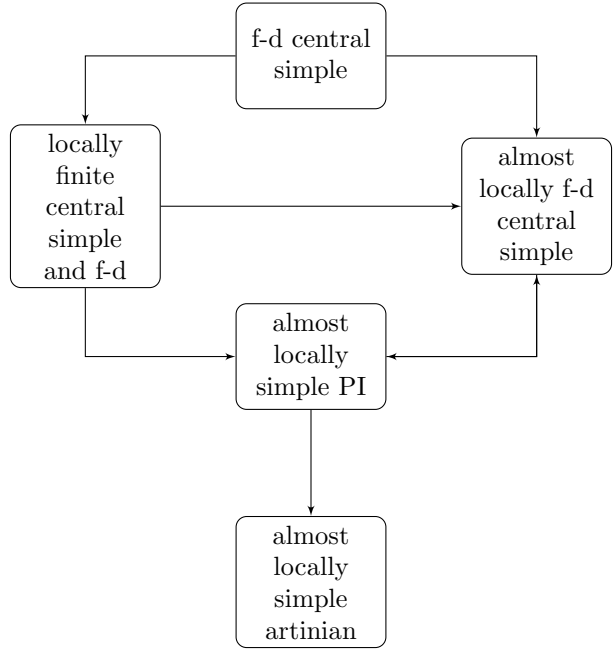
			\begin{tikzpicture}[node distance = 2cm, auto]
				
				\node [block] (A) {f-d central simple};
				\node[empty] (O) [below of=A] {};
				\node [block, left of=O, node distance=3cm] (B) {locally finite central simple and f-d};
				\node [block, right of=O, node distance=3cm] (C) {almost locally f-d central simple};
				\node [block, below of=O] (D) {almost locally simple PI};
				\node [block, below of=D, node distance=3cm] (E) {almost locally simple artinian};
				
				\path [line] (A) -| (B);
				\path [line] (A) -| (C);
				\path [line] (B) -- (C);
				\path [line] (B) |- (D);
				\path [line] (C) |- (D);
				\path [line] (D) -| (C);
				\path [line] (D) -- (E);
			\end{tikzpicture}
			
			\bigskip 
			\textbf{Important:} Note that in this chart f-d means finite-dimensional.	
		\end{center}
	\section{Involutions on rings and their restrictions}

	Let $A$ be a ring with involution $\star$ and $B$ a subring of $A$. Then, the restriction of $\star$ to $B$ not need to be an involution on $B$. However, in some situation, this is the case. For further use, in this section, we give some conditions on $\star$ and $B$ for which the restriction of $\star$ to $B$ is still an involution on $B$. Firstly, let us record two useful  remarks which recall the constructions of a subring and a division subring of a division ring generated by a subset. 
	\begin{remark}\label{remark_subring}
		Let $A$ be a ring and $S$ be a non-empty subset of $A$. Then, it is elementary to see that each element of the subring $\overline{S}$  of $A$ generated by $S$ can be written in a finite sum of monomials in elements taken from $S\cup\{1\}$; that is,
		\[
		\overline{S}=\left\{\sum\limits_{i=1}^n k_ia_{i1}a_{i2}\cdots a_{ir_i}|k_i\in \mathbb{Z}, a_{ij}\in S\cup\{1\}, n\in\mathbb{N}\right\}.
		\]
	\end{remark}
	
	\begin{remark}\label{remark_division subring}
		Let $D$ be a division ring with the prime subfield $\mathbb{P}$, and $S$ a non-empty subset of $D$. Then, the construction of the division subring $\mathbb{P}(S)$ of $D$ generated by $S$ over $\mathbb{P}$ is given as follows. Let $\overline{S}$ be the subring of $D$ generated by $S$. Define $L_0=\overline{S}$ and $L_0^{-1}=\{x^{-1}\mid x\in L_0\setminus \{0\}\}$. Inductively, we also define the following subsets of $D$:
		\[
		L_k=\overline{L_{k-1}\cup L^{-1}_{k-1}} \;\;\;\text{and}\;\;\;L^{-1}_k=\{x^{-1}\mid x\in L_k\setminus \{0\}\},
		\]
		for any $k\geq1$. Clearly, $L_0\subseteq L_1\subseteq L_2\subseteq \dots$, and it is straightforward to check that $\mathbb{P}(S)=\bigcup\limits_{k=0}^{\infty} L_k$.
	\end{remark}
	
	\begin{proposition}\label{propostion_restrict subring}
		Let $A$ be a ring with involution $\star$, and $S$ be a non-empty subset of $A$ such that $S^\star\subseteq S$. Then, the restriction of $\star$ to the subring $\overline{S}$ is an involution on $\overline{S}$.
	\end{proposition}
	\begin{proof}
		It is enough to prove that $\overline{S}^\star\subseteq \overline{S}$. For each element $x\in \overline{S}$, it follows from Remark \ref{remark_subring} that
		\[
		x=\sum\limits_{i=1}^n k_ia_{i1}a_{i2}\cdots a_{ir_i},
		\]
		where $k_i\in \mathbb{Z}, a_{ij}\in S\cup\{1\}, n\in\mathbb{N}$. Since $a_{i1}^\star,a_{i2}^\star,\ldots,a_{ir_i}^\star$ are in $S$, it follows that
		\[
		x^\star=\sum\limits_{i=1}^n k_ia_{ir_i}^\star \cdots a_{i2}^\star a_{i1}^\star\in\overline{S}.
		\]
		Hence, $\overline{S}^\star\subseteq \overline{S}$, as desired.
	\end{proof}
	
	\begin{proposition}\label{propostion_restrict division subring}
		Let $D$ be a division ring with an involution $\star$, $S$ a non-empty subset of $D$ such that $S^\star\subseteq S$, and $\mathbb{P}$  the prime subfield of $D$. Then, the restriction of $\star$ to $\mathbb{P}(S)$ is an involution on $\mathbb{P}(S)$.
	\end{proposition}
	\begin{proof}
		In what follows, we use the notation and construction of $\mathbb{P}(S)$ as in Remark~\ref{remark_division subring}.
		Thus, we have 
		\[
		\mathbb{P}(S)=\bigcup\limits_{k=0}^{\infty} L_{k}.
		\]
		In view of Proposition \ref{propostion_restrict subring}, to prove that the restriction of $\star$ to $\mathbb{P}(S)$ is an involution on $\mathbb{P}(S)$, we have to show that 
		$\mathbb{P}(S)^\star\subseteq \mathbb{P}(S)$.
		Clearly, it suffices to prove that $L_k^\star\subseteq L_k$ for every $k\in\mathbb{N}$. 
		If $k=0$, then $L_0=\overline{S}$, $L_0^{-1}=\{x^{-1}\mid 0\ne x\in \overline{S}\}$, so by hypothesis, it follows that $L_0^\star\subseteq L_0$. Assume by induction that $L_k^\star\subseteq L_k$. We have
		$$L_{k+1}=\overline{L_k\cup L^{-1}_k}.$$
In view of Proposition \ref{propostion_restrict subring}, $L_k$ is a ring with involution $\star$.  So, for any $x\in L_k$, we have
		\[
		(x^{-1})^\star=(x^\star)^{-1}\in L_0^{-1},
		\]
		which implies that $(L_k^{-1})^\star\subseteq L_k^{-1}$, and consequently, 
		$L_k^\star\cup (L_k^{-1})^\star\subseteq L_k\cup L_k^{-1}$ or $(L_k\cup L_k^{-1})^\star\subseteq L_k\cup L_k^{-1}$, from which it follows that $L_{k+1}^\star\subseteq L_{k+1}$. By induction principle, we conclude that $L_k^\star\subseteq L_k$ for every $k\in\mathbb{N}$ as desired.
	\end{proof}
	
	\begin{lemma}\label{lemma_restriction}
		Let $A$ be a ring with center $Z$, $B$ a subring of $A$ containing $Z$, and  $\star$  an involution on $A$.  If $T_B\subseteq Z$, then the restriction of $\star$ to $B$ is also an involution on $B$.
	\end{lemma}
	\begin{proof}
		In view of Proposition \ref{propostion_restrict subring}, it suffices to prove that $b^\star\in B$ for any $b\in B$. Indeed, since $T_B\subseteq Z$, we have $b^\star + b\in Z\subseteq B$. Hence, $b^\star+b=b'$ for some $b'\in B$, which implies that $b^\star=b'-b\in B$. 
	\end{proof}

\section{Skew linear groups with involutions}
	Before proceeding further, let us recall some fundamental facts about finite-dimensional central simple algebras. Let $A$ be a finite-dimensional central simple algebra over a field $F$. Then, by the well-known Wedderburn-Artin Theorem, there is a unique integer $r$ and a  unique up to isomorphism  division ring $D$ with center $F$ such that $A\cong {\rm M}_r(D)$. It is well-known that the dimension of $D$ over $F$ is a square; that is, $[D:F]=k^2$, for some $k\geq1$. It follows that $[A:F]=r^2k^2$; and hence $[A:F]$ is also a square. Accordingly, the \textit{degree} of $A$ is defined to be ${\rm deg}(A)=\sqrt{[A:F]}$. Let $A$ be such an $F$-algebra with involution $\star$. As $F$ is invariant under $\star$, it can be checked that the set of traces $T_A=\{x+x^\star\;|\; x\in A\}$ of $A$ is a $F$-subspace of $A$, while the set $N_A$ is not. The dimension $[T_A:F]$ of $T_A$ over $F$ is given in the following lemma.
	 
\begin{lemma}\label{propostion_dimensions}
		Let $A$ be a finite-dimensional central simple $F$-algebra of degree $m$ with involution $\star$, and $d=[T_A:F]$. Then, either $d=\frac{m(m+1)}{2}$ or $d=\frac{m(m-1)}{2}$.
\end{lemma}
\begin{proof}
	This lemma is an immediate corollary of \cite[Proposition 2.6]{knus-merkurjev-rost-tignol-98}.
\end{proof}
	
		In view of this lemma, we see that the $F$-subspace $T_A$ is too big in $A$. As, we have mentioned in the Introduction, the susets $T_A$ and $N_A$ both are contained in $A_+$. Although $N_A$ is not an $F$-subspace of $A$, we also prove that if $N_A\subset Z(A)$, then $T_A\subset Z(A)$. Indeed, for any $x\in A$, we have 
		$$
		(1+x)(1+x)^\star=(1+x)(1+x^\star)=1+x+x^\star+xx^\star.
		$$
		Because both $(1+x)(1+x)^\star$ and $xx^\star$ are contained in $Z(A)$, we get $x+x^\star\in Z(A)$, yielding that $T_A\subseteq Z(A)$. Thus, the implication $N_A\subseteq Z(A)\Rightarrow T_A\subseteq Z(A)$ indicates that the set $N_A$, although is not an $F$-subspace of $A$,  is also quite big in $A$. In the next corollary, we demonstrate that the fact $T_A\subseteq Z(A)$ imposes a strong impact on the structure of the whole $A$. 
	
	\begin{corollary}\label{corollary_centrally finite case}
		Let $D$ be a centrally finite division ring with center $F$ and $n\geq1$. Let $\star$ be an involution on ${\rm M}_n(D)$. If $T_{{\rm M}_{n}(D)}\subseteq F$, then either 
		\begin{enumerate}[font=\normalfont]
			\item[(i)] $n=1$, and $D=F$ or $D$ is  a quaternion division algebra over $F$, or
			\item[(ii)] $n=2$, and $D=F$ and $\star$ is the ordinary symplectic involution on ${\rm M}_2(F)$; that is, the involution $\star$ is given by:
			$$
			\begin{pmatrix}
				a & b \\
				c & d 
			\end{pmatrix}^\star =
			\begin{pmatrix}
				d & - b \\
				- c & a 
			\end{pmatrix}.
			$$
		\end{enumerate}
	\end{corollary}
	\begin{proof}
		Let $[D:F]=k^2$, for some integer $k\geq1$. It follows that $A:={\rm M}_n(D)$ is a $F$-central simple algebra of degree $m:=kn$. In view of Lemma \ref{propostion_dimensions}, there are two possible cases:
		
		\smallskip 
		
		\textbf{\textit{Case 1.}} $[T_A:F]=\frac{m(m+1)}{2}$. In this case, the involution $\star$ must be of orthogonal type. Moreover, as $T_A\subseteq F$, we get that $\frac{m(m+1)}{2}=1$, from which it follows that $m=1$. This means that $n=k=1$, and so $A=D=F$. 
		
		\smallskip 
		
		\textbf{\textit{Case 2.}} $[T_A:F]=\frac{m(m-1)}{2}$.  As in Case 1, the condition $T_A\subseteq F$ implies that $\frac{m(m-1)}{2}=1$. It follows that $m=2$, which means that $n=2$ and $k=1$, or else $n=1$ and $k=2$. 
		
		\smallskip 
		
		\textit{Case 2.1.} $n=2$ and $k=1$. In this case, we have $D=F$ and so $A\cong {\rm M}_2(F)$. By \cite[Theorem 2.3]{Chacron.2016a.paper}, the involution $\star$ is scalar.
			
			Assume that 
			$\begin{pmatrix}
				a & b \\
				c & d 
			\end{pmatrix}\in A$ and 
			$\begin{pmatrix}
				a & b \\
				c & d 
			\end{pmatrix}^\star=\begin{pmatrix}
				m & n \\
				p & q 
			\end{pmatrix}$.
			Since $\star$ is a symplectic and scalar, it follows that 
			$\begin{pmatrix}
				a & b \\
				c & d 
			\end{pmatrix}+\begin{pmatrix}
			a & b \\
			c & d 
		\end{pmatrix}^\star\in F$, and $\begin{pmatrix}
		a & b \\
		c & d 
	\end{pmatrix}\begin{pmatrix}
	a & b \\
	c & d 
\end{pmatrix}^\star\in F$. Therefore,

			$$\begin{pmatrix}
				a+m & b+n \\
				c+p & d+q 
			\end{pmatrix} \in F$$ and 
			$$\begin{pmatrix}
				am-bc & -ab+bq \\
				cm-cd & -cb+dq 
			\end{pmatrix}\in F,$$
		which implies that $b+n=c+p=-ab+bq=cm-cd=0$. Hence, $n=-b, p=-c, q=a, m=d$, and we have
			$$
			\begin{pmatrix}
				a & b \\
				c & d 
			\end{pmatrix}^\star =
			\begin{pmatrix}
				d & - b \\
				- c & a 
			\end{pmatrix}.
			$$
		
		\smallskip 
		
		\textit{Case 2.2.} $n=1$ and $k=2$. It follows that $A=D$ and $[D:F]=4$, which means that $A$ is a quaternion algebra over $F$. The proof of the corollary is now complete.
	\end{proof}
	The following lemma follows immediately from \cite[Theorem 2.3]{Chacron.2016a.paper}.
	
	\begin{lemma}\label{lemma_semi-prime}
		Let $A$ be an unital semi-prime ring with center $Z$, and $\star$ be an involution on $A$. Then, $T_A\subseteq Z$ if and only if $N_A\subseteq Z$. 
	\end{lemma}

	\begin{lemma}\label{lemma_centrally finite}
		Let $D$ be a division ring with center $F$,  $n$ a positive integer, and  $\star$ an involution on ${\rm M}_{n}(D)$. If $T_{{\rm M}_{n}(D)}\subseteq F$, then $D$ is centrally finite. Consequently, either assertion (i) or (ii) of Corollary~\ref{corollary_centrally finite case} holds for $D$.
	\end{lemma}
	\begin{proof}
		In view of Lemma \ref{lemma_semi-prime}, the involution $\star$ is both symplectic and scalar.  Hence, for each $a\in {\rm M}_n(D)$, the elements $s:=a^\star+a$ and $p:=a^\star a$ are both contained in $F$. It is straightforward to check that $a^2-sa+p=0$. This shows that the matrix ring ${\rm M}_n(D)$ is algebraic of bounded degree $2$, and so is $D$. By the famous result of Jacobson \cite[Theorem 7]{Pa_Jacobson_1945}, $D$ is finite dimensional over $F$ as desired. 
	\end{proof}

	Next, we restrict the conditions for the involution to some certain subgroup of ${\rm GL}_n(D)$. Firstly, we consider the case $n=1$; that is, the case of ${\rm GL}_1(D)=D^\times$.
	\begin{lemma}\label{lemma_restriction to group}
		Let $R$ be a ring with involution $\star$, $G$  a subgroup of $R^\times$. Let $Z[G]$ be the subring of $R$ generated by $G$ over the center $Z$ of $R$. If $T_G\subseteq Z$, then the restriction of $\star$ to $Z[G]$ is a symplectic involution on $Z[G]$. 
	\end{lemma}
	
	\begin{proof}
		Take an arbitrary element $x\in Z[G]$, and write it in the form
		$$
		x= a_1g_1+a_2g_2+\cdots+a_ng_n,
		$$
		where $a_i\in Z$ and $g_i\in G$. Because $\star$ leaves fixed elements in $Z$, it follows that 
		$$
		x+x^\star=a_1(g_1+g_1^\star)+a_2(g_2+g_2^\star)+\cdots+a_n(g_n+g_n^\star). 
		$$
		As $T_G\subseteq Z$, we conclude that $g_i+g_i^\star\in Z$ for all $i\in \{1,\dots,n\}$, and so $x+x^\star\in Z$, which implies that $x^\star\in Z[G]$. Hence, the restriction of $\star$ to $Z[G]$ is also an involution on $Z[G]$ which is clearly symplectic. 
	\end{proof}
	
	\begin{lemma}\label{lemma_invariant}
		Let $D$ be a division ring with center $F$, and $H$ a subset of $D^\times$. For an element $g\in D$, let $F(g)$ be the subfield of $D$ generated by $g$ over $F$. If $h^{-1}gh\in F(g)$ for any $h\in H$, then $F(g)$ is $H$-invariant.
	\end{lemma}
	\begin{proof}
		Every element in $F(g)$ can be written in the form $a(g)b(g)^{-1}$, where $a(t),b(t)$ are two polynomials in $F[t]$ such that $b(g)\ne0$. For any $0\ne h\in H$, we have 
		$$
		h^{-1}a(g)b(g)^{-1}h=(h^{-1}a(g)h)(h^{-1}b(g)^{-1}h)=(h^{-1}a(g)h)(hb(g)h^{-1})^{-1}.
		$$
		Write 
		$$
		\begin{aligned}
			& a(g)=a_0+a_1g+\cdots+a_ng^n,\\
			& b(g)=b_0+b_1g+\cdots+b_mg^m,
		\end{aligned}
		$$
		where $a_i,b_j\in F$ for all $0\leq i\leq n$ and $0\leq j\leq m$. Because $h^{-1}gh\in F(g)$, it easily checked that the elements  $h^{-1}a(g)h$ and $hb(g)h^{-1}$ belong to $F(g)$. It follows that $h^{-1}a(g)b(g)^{-1}h\in F(g)$, and the proof of the lemma is proved.
	\end{proof}
	
	\begin{lemma}\label{lemma_restriction to group_scalar}
		Let $D$ be a division ring with center $F$,  $G$  non-central subnormal subgroup of $D^\times$, and $\star$  an involution on $D$. Then, $N_G\subseteq F$ if and only if  $\star$ is scalar.
	\end{lemma}
	\begin{proof}
		The ``if'' part is clear. Now, assume that $N_G \subseteq F$. To prove $N_D\subseteq F$, it suffices to show that $N_{D^\times} \subseteq F$. Because $G$ is subnormal in $D^\times$, there exist a smallest integer $r\geq1$ and a series of subgroups
		$$
		G=G_r\unlhd G_{r-1}\unlhd\cdots\unlhd G_1\unlhd G_0=D^\times,
		$$
		in which $G_i$ is normal in $G_{i-1}$ for all $1\leq i\leq r$. For each $i$, we claim that if $N_{G_i}\subseteq F$ then $N_{G_{i-1}}\subseteq F$. Assume by contrary that $N_{G_{i-1}}\not\subseteq  F$. Fix an element $g\in G_i\backslash F $. For any $x\in G_{i-1}$, since $x^{-1}gx\in G_i$, we have 
		$$
		(x^{-1}gx)(x^{-1}gx)^\star= x^{-1}gxx^\star g^\star (x^\star)^{-1}\in F.
		$$
		Therefore $ x^{-1}gxx^\star g^\star (x^\star)^{-1}=a$, for some $a\in F$, which implies  that 
		$$gxx^\star g^\star=a xx^\star.\eqno (1)$$ Since $N_{G_i}\subseteq F$ and $g\in G_i$, there exists some element $b\in F$ such that $gg^\star=b$. Replace  $g^\star = bg^{-1}$ in (1), we get  $bgxx^\star g^{-1}=a xx^\star$, which implies that $$
		(xx^\star)^{-1}gxx^\star=b^{-1}ag\in F(g).
		$$
		Because $x$ was chosen arbitrarily in $G_{i-1}$, the last equation shows that the subfield $F(g)$ of $D$ is $N_{G_{i-1}}$-invariant by Lemma \ref{lemma_invariant}. Moreover, as $G\subseteq G_{i-1}$,  it is easy to check that $g N_{G_{i-1}} g^\star\subseteq N_{G_{i-1}}$ for any $g\in G$ and that $1\in N_{G_{i-1}}$. According to \cite[Lemma 2.1(ii)]{bien-hai-hue}, we get that  either $F(g)\subseteq F$ or $F(g)=D$. If the first case occurs, we have $g\in F$, which is impossible. The second case is impossible because $D$ is non-commutative. Thus, the claim is shown. Now, by   induction on $r$, we conclude that $N_D\subseteq F$, and so $\star$ is scalar. 
	\end{proof}
	
	\begin{proposition}\label{propostion_involution on subgroup}
		Let $D$ be a division ring with center $F$, and $G$ a subgroup of $D^{\times}$ such that $F(G)=D$. Let $\star$ be an involution on $D$. If $T_G\subseteq F$, then $D$ is a centrally finite division ring. 
	\end{proposition}
	\begin{proof}
		Assume that $T_G\subseteq F$. Let $F[G]$ be the subring of $D$ generated by $G$ over $F$. According to Lemma \ref{lemma_restriction to group}, we conclude that the restriction of $\star$ to $F[G]$ is an symplectic involution on $F[G]$. Hence, by \cite[Theorem 2.3]{Chacron.2016a.paper}, the restriction of $\star$ to $F[G]$ is scalar. So, for each $a\in F[G]$, we have $aa^\star\in Z(F[G])\subseteq C_D(G)$, where $C_D(G)$ is the centralizer of $G$ in $D$. Since $F(G)=D$, we  have $C_D(G)=C_D(D)=F$, and so, $aa^\star\in F$. If we set $s:=a^\star+a$ and $p:=a^\star a$, then $s$ and $p$ are both in $F$. It is straightforward to check that $a^2-sa+p=0$, from which it follows that $a$ satisfies the polynomial $x^2-sx+p\in F[x]$. This shows that the prime ring $F[G]$ is algebraic of bound degree $2$ over the field $F$, and by \cite[Theorem~ 3]{Pa_Kaplansky_1948}, $F[G]$ satisfies a polynomial identity. By \cite[Lemma 1]{amitsur}, $F[G]$ has a division ring of quotients, consisting of all elements of the form $sr^{-1}$, where $r,s\in F[G]$ and $r\ne0$, which coincides with $F(G)=D$. According to \cite[Theorem 1]{amitsur}, we conclude that $D=F(G)$ satisfies a polynomial identity, and so $[D:F]<\infty$ by \cite[Theorem~1]{Pa_Kaplansky_1948}. 
	\end{proof}

We are now ready to prove the main theorem of this section, which broadly generalizes \cite[Theorem 6.1]{bien-hai-hue}.
	\begin{theorem}\label{theorem_subnormal skew linear group}
		Let $D$ be a division ring with center $F$ which contains at least four elements, $n$ a positive integer, and $\star$ an involution on $A:={\rm M}_n(D)$. If $G$ is a non-central subnormal subgroup of ${\rm GL}_n(D)$, then the following assertions are equivalent:
		\begin{enumerate}[font=\normalfont]
			\item $T_G\subseteq F$;
			\item $T_A\subseteq F$;
			\item $N_A\subseteq F$;
			\item $N_G\subseteq F$.
		\end{enumerate}
		 Moreover, if one of these conditions holds, then either
		\begin{enumerate}[font=\normalfont]
			\item[(i)] $n=1$ and $D$ is  a quaternion division algebra over $F$, or
			\item[(ii)] $n=2$, $D=F$ and $\star$ is the ordinary symplectic involution on ${\rm M}_2(F)$.
		\end{enumerate}
	\end{theorem}
	\begin{proof}
		We divide our situation into two possible cases:
		
		\smallskip 
		
		\textbf{\textit{Case 1:}} $n=1$.
		
		 In this case we have $A=D$. Firstly, we show that  (1) $\Longleftrightarrow$ (2). It is clear that (2) implies (1).  Now, assume  $T_G\subseteq F$. Because $G$ is non-central subnormal subgroup of $D^\times$,  it follows from Stuth's Theorem (see \cite[Theorem 1]{Pa_Stuth_1964}) that $F(G)=D$, and so $[D:F]<\infty$ by Proposition \ref{propostion_involution on subgroup}. By \cite[Lemma 2.3]{Pa_hai-thin_2009}, it follows that $D=F(G)=F[G]$, and so $\star$ is a symplectic involution on $D$ by Lemma \ref{lemma_restriction to group}. This implies that $T_D\subseteq F$, and so (2) holds.
		 
		 As $A$ is a simple ring, the equivalence (2) $\Longleftrightarrow$ (3) follows from Lemma \ref{lemma_semi-prime}. The implication (3) $\Longrightarrow$ (4) follows immediately from the fact that $N_G\subseteq N_A$. Finally, the implication (4) $\Longrightarrow$ (3) follows from Lemma \ref{lemma_restriction to group_scalar}. We have just proved the equivalences of (1), (2), (3) and (4) for the case $n=1$. To prove (i), assume that (2) holds. Then, by Proposition \ref{propostion_involution on subgroup}, we get $[D:F]<\infty$. Thus, the assertion (i) follows from Corollary \ref{corollary_centrally finite case}.
		
		\smallskip 
		
		\textbf{\textit{Case 2:}} $n>1$.
		
		We prove the equivalences of (1), (2), (3), (4) for this case. Firstly, not that the implication (2) $\Longrightarrow$ (1) is trivial. To prove the implication (1) $\Longrightarrow$ (2), we first claim that  $F[G]={\rm M}_n(D)$.  Indeed, by \cite[Theorem 4]{Pa_Mahdavi-Hezavehi_1998},  ${\rm SL}_n(D)\subseteq G$ because $G$ is a non-central subnormal subgroup of ${\rm GL}_n(D)$. Hence, $F[G]$ contains $F[{\rm SL}_n(D)]$ which is normalized by ${\rm GL}_n(D)$. According to \cite[Corollary 1]{Pa_Rosenberg_1956}, we conclude that $F[G]=F[{\rm SL}_n(D)]={\rm M}_n(D)$. Now, if (1) holds, then $T_G\subseteq F$, and by Lemma \ref{lemma_restriction to group},  $\star$ is a symplectic involution on ${\rm M}_n(D)$, and so (2) follows. Thus, we have proved the equivalence of (1) and (2). The proof of the equivalences of the remaining assertions in this case is similar as in  \textbf{\textit{Case 1}}.
		
		For the proof of (ii), assume that (2) holds. Firstly, we prove that $D=F$. Assume by contrary that $D$ is non-commutative, which implies that $D$ is infinite. We have shown above that  $F[G]={\rm M}_n(D)$. Now, with a reference to Lemma \refeq{lemma_restriction to group}, we conclude that the involution $\star$ on ${\rm M}_n(D)$ is symplectic, and so by Lemma \ref{lemma_centrally finite}, it follows that $D$ is a field, a contradiction. Therefore, we have $D=F$. By applying Corollary \ref{corollary_centrally finite case}, we must have $n=2$, and $\star$ is the ordinary symplectic involution on ${\rm M}_2(F)$.
	\end{proof}
	\begin{corollary}\label{corollary_subnormal skew linear group}
		Let $D$ be a division ring with center $F$ which contains at least four elements, $n$ a positive integer, $G$ a non-central subnormal subgroup of ${\rm GL}_n(D)$, and $\star$  an involution on ${\rm M}_n(D)$. Then, the conditions ${\rm (C1)}$ -- ${\rm (C5)}$ are all equivalent for ${\rm M}_n(D)$, and these conditions are also equivalent to the followings:
		\begin{enumerate}[font=\normalfont]
		\item$x+x^\star$ is central for all $x\in G$.
		\item $xx^\star$ is central for all $x\in G$.
	\end{enumerate}
	Moreover, if one of these assertions holds, then either 
			\begin{enumerate}[font=\normalfont]
		\item[(i)] $n=1$ and $D$ is  a quaternion division algebra over $F$, or
		\item[(ii)] $n=2$, $D=F$ and $\star$ is the ordinary symplectic involution on ${\rm M}_2(F)$.
	\end{enumerate}
	\end{corollary}
	\begin{proof}
		As ${\rm M}(D)$ is a simple ring, the equivalences of (C1)--(C5) follow form \cite{Chacron.2017a.paper}  and \cite{Chacron.2016a.paper}. The conditions (1) and (2) are respectively equivalent to $T_G\subseteq F$ and $N_G\subseteq F$. Thus, all remaining assertions follow from Theorem \ref{theorem_subnormal skew linear group}.
	\end{proof}
	
	\section{Almost locally simple artinian algebras with involution}
	In this section, we try to extend Theorem \ref{theorem_subnormal skew linear group} for almost locally simple artinian rings. The main result of this section is Theorem \ref{theorem_main-locally simple artin} below, in which the subnormal subgroup $G$ must be assumed to be non-abelian, rather than being non-central as in Theorem \ref{theorem_subnormal skew linear group}. 
	
	\begin{proposition}\label{simplicity}
		Every almost locally simple ring is simple. 
	\end{proposition}
	\begin{proof}
		Let $A$ be an almost locally simple ring. Assume that $I$ is a non-zero ideal of $A$, and $0 \ne a \in I$. Then, there exists a simple subring $A_0$ of $A$ such that $a\in A_0$. It follows that $a\in I_0:= A_0\cap I$, which is an ideal of $A_0$. Since $A_0$ is simple and
		$I_0\ne 0$, we conclude that $I_0=A_0$ from which it follows that $1\in I_0\subseteq I$.  This
		implies that $I=A$, and so $A$ is simple.
	\end{proof}
	
	\begin{proposition}\label{theorem_locally simple artin}
		Let $A$ be an  almost locally simple artinian algebra over its center $F$, and  $\star$  an involution on $A$. If $T_A\subseteq F$, then one of the following assertions hold:
		\begin{enumerate}[font=\normalfont]
			\item[(i)] $A=F$, or $A$ is a quaternion division algebra over its center $F$, or
			\item[(ii)]$A\cong{\rm M}_2(F)$ and $\star$ is the ordinary symplectic involution on $A$.
		\end{enumerate}
	\end{proposition}
	\begin{proof}
		 In view of Proposition \ref{simplicity}, $A$ is simple, and hence $A$ is a prime ring. Let $S$ be a finite subset of $A$. Then, the $F$-subalgebra of $A$ generated by $S$ is contained in some simple artinian algebra, say $B$. It follows that $B\cong {\rm M}_k(E)$, where $k\geq 1$ and $E$ is a  division ring. As $F\subseteq Z(B)$, by Lemma \ref{lemma_restriction}, the restriction of $\star$ to $B$ is also an involution on $B$ which is certainly symplectic. According to Proposition  \ref{propostion_involution on subgroup}, we conclude that $E$ is a centrally finite division ring. Hence, in view of Corollary \ref{corollary_centrally finite case}, it is easy to conclude that there are three possibilities for $B$:  $B$ is a field; $B$ is contained in the matrix ring $\mathrm{M}_4(Z(B))$; and $B=\mathrm{M}_2(Z(B))$. In view of the Armitsur-Levitzki Theorem (see \cite[Theorem~1]{amitsur-levitzki}), $B$ satisfies the standard polynomial identity $S_{2n}=0$ for $n=1,2$ or $4$. Being a subset of $B$, $S$ satisfies the same polynomial identity as $B$. Since $S$ was taken arbitrarily in $A$, we can conclude that $A$ satisfies the polynomial identity $S_2S_4S_8=0$. By Kaplansky's Theorem (see \cite[Theorem 1]{Pa_Kaplansky_1948}), it follows that $A$ is a simple artinian ring. According to the Wedderburn-Artin Theorem, $A\cong {\rm M}_n(D)$, for some $n\geq1$ and some centrally finite division ring $D$.  Finally, all conclusions follows from Corollary \ref{corollary_centrally finite case}.
	\end{proof}
		
	\begin{theorem}\label{theorem_main-locally simple artin}
			Let $A$ be an almost locally simple artinian algebra with center $F$ which contains at least four elements, $G$ be an non-abelian subnormal subgroup of $A^\times$, and $\star$  an involution on $A$. Then, the following assertions are equivalent:
			\begin{enumerate}[font=\normalfont]
				\item $T_G\subseteq F$;
				\item $T_A\subseteq F$;
				\item $N_A\subseteq F$;
				\item $N_G\subseteq F$.
			\end{enumerate}
			Moreover, if one of these conditions holds, then $A\cong\mathrm{M}_n(D)$ for some positive integer $n$, and some centrally finite division ring $D$; and  either
			\begin{enumerate}[font=\normalfont]
				\item[(i)] $n=1$, and $A$ is  a quaternion division algebra over $F$, or
				\item[(ii)] $n=2$, and $A=F$ and $\star$ is the ordinary symplectic involution on ${\rm M}_2(F)$.
			\end{enumerate}
	\end{theorem}
	\begin{proof}
		Let $S$ be an arbitrary finite subset of $A$. Fix $a,b\in G$ such that $ab\ne ba$.  Then, as we have shown in the proof of Theorem \ref{theorem_locally simple artin}, we conclude that $F$-subalgebra of $A$ generated by $S\cup \{a,b\}$ is contained in a matrix ring ${\rm M}_k(E)\subseteq A$, where $k\geq 1$ and $E$ is a division ring. Also, we have $F\subseteq Z(E)$, and the restriction of $\star$ to ${\rm M}_k(E)$ is also an involution on ${\rm M}_k(E)$. Let $G_1=G\cap {\rm GL}_k(E)$. Then, $G_1$ is a subnormal subgroup of ${\rm GL}_k(E)$. Because $a,b\in G_1$, it follows that $G_1$ is not contained in the center of ${\rm M}_k(E)$, which should be identified with $Z(E)$. Now, assume that (1) holds. Then, we have $g+g^\star\in F\subseteq  Z(E)$ for all $g\in G$. In particular, $h+h^\star\in Z(E)$ for all $h\in G_1$. Hence, we can apply Theorem \ref{theorem_subnormal skew linear group} for the matrix ring ${\rm M}_k(E)$ to conclude that $E$ is finite dimensional over its center $Z(E)$. By the same arguments used in the proof of Proposition \ref{theorem_locally simple artin}, we get that $A\cong {\rm M}_n(D)$, for some $n\geq1$ and some centrally finite division ring $D$.  Thus, the assertions (i) and (ii) follow from Theorem \ref{theorem_subnormal skew linear group}.
	\end{proof}
	\begin{corollary}\label{corollary_locally simple artin}
		Let $A$ be an almost locally simple artinian algebra with center $F$ which contains at least four elements, $G$  an non-abelian subnormal subgroup of $A^\times$, and $\star$  an involution on $A$. Then, the conditions ${\rm (C1)}$ -- ${\rm (C5)}$ are equivalent for $A$, and these conditions are also equivalent to the following conditions:
		\begin{enumerate}[font=\normalfont]
			\item$x+x^\star$ is central for all $x\in G$.
			\item $xx^\star$ is central for all $x\in G$.
		\end{enumerate}
		Moreover, if one of these assertions holds, then either 
		\begin{enumerate}
			\rm\item[(i)]\textit{$A$ is a quaternion division algebra over its center $F$, and $\star$ is the symplectic involution, or}
			\rm\item[(ii)]\textit{$A\cong{\rm M}_2(F)$ and $\star$ is the ordinary symplectic involution on $A$.}
		\end{enumerate}
	\end{corollary}
	\begin{proof}
		The proof is similar to that of Corollary \ref{corollary_subnormal skew linear group}.
	\end{proof}
We close the paper with an interesting question which we shall investigate in a near future.

\begin{remark}
	Let $D$ be a division ring with involution $\star$ of the first kind, and $K$ an arbitrary subfield of $D$ which is unnecessary invariant under $\star$. If $T_D\cup N_D\subseteq K$, then every element $a\in D$ satisfies the quadratic equation $x^2-sx+p=0$, where $s=a+a^\star$ and $p=a^\star a$. This means that $D$ is left algebraic of bound degree $2$ over $K$. According to \cite[Theorem 1.3]{bell}, we get that $D$ is a centrally finite division ring with $[D:Z(D)]\leq 4$; that is, $D$ is either a field or a quaternion division ring. At this point, it is reasonable to ask if the main results of the current paper should be  also true if we replace the center of $D$ by an arbitrary its subfield ?
\end{remark}
	
\end{document}